\documentclass[12pt]{article}

         \usepackage[reqno, namelimits, sumlimits]{amsmath} 
         \usepackage{amssymb,  amsfonts, amsbsy}
        \usepackage{amsthm} 
        \usepackage[english]{babel}

\usepackage{enumitem}

       \usepackage{graphicx} 

\newtheorem{Theorem}{Theorem}[section]

\newtheorem{theorem}[Theorem]{Theorem}
\newtheorem{lemma}[Theorem]{Lemma}

\newtheorem{remark}[Theorem]{Remark}
\newtheorem{Remark}[Theorem]{Remark}

\newcommand{\R}{{\mathbb R}}
\newcommand{\N}{{\mathbb N}}

 \newcommand{\cV}{{\mathcal V}}

\newcommand{\al}{\alpha}
\newcommand{\be}{\beta}
\newcommand{\ga}{\gamma}

\newcommand{\de}{\delta}

\newcommand{\om}{\omega}

\newcommand{\Om}{\Omega}

\newcommand{\la}{\lambda}

\newcommand{\codt}{\cdot} 

\usepackage{color} 

\def\e{\varepsilon}
\def\a{\alpha}
\def\b{\beta}
\def\la{\lambda}
\def\d{\delta}

\def\vp{\varphi}

\def\lp{\left(}
\def\rp{\right)}

\def\lb{\left|}
\def\rb{\right|}
\def\lV{\left\Vert}
\def\rV{\right\Vert}

\def\ME{\mathcal{E}}
\def\MO{\mathcal{O}}
\def\MP{\mathcal{P}}

\newtheorem{prop}[Theorem]{Proposition}
 \newtheorem{lem}[Theorem]{Lemma}
 \newtheorem{cor}[Theorem]{Corollary}

\numberwithin{equation}{section}

\title{Large-time behavior of solutions of parabolic 
equations on the real line with convergent initial data}
\author{Antoine Pauthier and Peter Pol\'{a}\v{c}ik\footnote{Supported in part by the NSF
    Grant DMS-1565388}  \\
{\small School of Mathematics, University of Minnesota}\\
{\small Minneapolis, MN 55455}
}

\date{}
 
\begin{document}
\maketitle
\begin{abstract}
 We consider the semilinear parabolic equation $u_t=u_{xx}+f(u)$ on
 the real line, where $f$ is a locally
  Lipschitz function on $\R.$ 
We prove that if a solution $u$ 
of this equation is  bounded
and its initial value $u(x,0)$ has distinct limits at 
$x=\pm\infty,$ then the solution is quasiconvergent, that is, all its
limit profiles as $t\to\infty$ are steady states.  
\end{abstract}


\section{Introduction}
 Consider the Cauchy problem 
 \begin{align}
 u_t=u_{xx}+f(u), & \qquad x\in\R,\ t>0, \label{eq1}\\
 u(x,0)=u_0(x), & \qquad x\in\R, \label{ic1}
 \end{align}
where $f$ is  a  locally
  Lipschitz  function  on $\R$ and
$u_0\in C_b(\R):=C(\R)\cap L^\infty(\R)$. 
  We denote by $u(\cdot,t,u_0)$ the unique
classical solution of (\ref{eq1})-(\ref{ic1}) 
and by $T(u_0)\in(0,+\infty]$ its maximal existence time. If $u$ is
bounded on $\R\times[0,T(u_0))$, then necessarily $T(u_0)=+\infty,$
that is, the solution is global. In this paper, we are concerned with
the behavior of bounded solutions as  $t\to\infty$. A basic question
we specifically want to address is whether, or to what extent, 
the large-time behavior of bounded solutions is governed by steady
states of \eqref{eq1}.  

This question has long  been settled for equation \eqref{eq1} 
 considered on a bounded interval, instead
of $\R$, and complemented by one of common boundary conditions, say Dirichlet,
Neumann, Robin, or periodic. Namely, in that case each bounded solution 
converges, uniformly on the spatial interval,  to a steady state 
\cite{Chen-M:JDE, Matano:conv, Zelenyak}.  In contrast, the large-time
behavior of equation \eqref{eq1} on $\R$ is not generally so 
simple and is much less understood. 

To talk about the behavior in more specific terms, recall that, 
by standard parabolic regularity estimates,
any bounded solution of \eqref{eq1} has relatively 
compact orbit in $L^\infty_{loc}(\R)$. In other words, any
sequence $t_n\to\infty$ has a subsequence $\{t_{n_k}\}$ such that 
$u(\cdot,t_{n_k})\to \varphi$ in $L^\infty_{loc}(\R)$ for some
continuous function $\varphi$. It is therefore natural to use the
topology of $L^\infty_{loc}(\R)$ when considering the 
convergence of solutions and related issues. Thus, we say that
a bounded solution $u$ is \emph{convergent} if for some $\varphi$ one has 
 $u(\cdot,t)\to\varphi$  locally uniformly on $\R$.
Of course, the convergence may take place in stronger topologies, but we take 
the convergence in $L_{loc}^\infty(\R)$, the topology in which  
the orbit is compact,  as a natural minimal
requirement. 
 
While the convergence of the solution of \eqref{eq1}, \eqref{ic1} has been
proved under various conditions on $u_0$ and $f$ 
\cite{Chen-L-Z-G,Du-M,p-Du, Fasangova, Fasangova-F,
  Feireisl:long-time, p-Fe, p-Ma:1d, Muratov-Z, P:unbal,Zlatos:sharp},
it is not the general behavior of bounded solutions
even when $f\equiv 0$, that is, when \eqref{eq1} is the linear heat
equation. As observed in \cite{Collet-E}, if $u_0$ takes values $0$
and $1$ on suitably spaced long intervals with
sharp transitions between them, then, as $t\to\infty$, $u(\codt,t)$ 
 approaches $0$ along a sequence of times $t_n\to
\infty$ and  $1$ along another such sequence 
(the convergence is in $L_{loc}^\infty(\R)$ in both cases). 

As we explain shortly, 
for the linear equation the large-time behavior of any bounded
solution is still governed by steady states in the sense that every
limit profile of any such solution is a steady state. 
Here a \emph{limit profile} of a bounded solution
$u$ of \eqref{eq1} refers to any element of the 
$\omega$-limit set of  $u$:
\begin{equation}\label{defomega}
 \omega(u):=\left\{ \vp\in C_b(\R):\ u(\cdot,t_n)\to\vp \textrm{
     for some sequence }t_n\to\infty\right\},
\end{equation}
where the convergence is in $L^\infty_{loc}(\R)$. 
If the solution $u$ corresponds to a given initial datum $u_0$, we
also write $\om(u_0)$ for $\om(u)$.   We say that a bounded solution 
$u$ of \eqref{eq1} is \emph{quasiconvergent} 
if $\om(u)$ consists entirely of steady states. 
Thus, a quasiconvergent solution approaches a set of steady states,
from which it follows that $u_t(\cdot,t)\to 0$, locally uniformly on
$\R$, as $t\to\infty$.  This makes  quasiconvergent solutions 
hard to distinguish---numerically, 
for example---from convergent solutions; they 
move very slowly at large times. 

In the case of the linear heat equation, 
the quasiconvergence of each bounded solution  
follows from the invariance property of  the
$\om$-limit set: $\om(u)$ consists 
of \emph{entire} solutions of $\eqref{eq1}$, by which we mean solutions
defined for all $t\in \R$. If $u$ is bounded, then
the entire solutions in $\om(u)$ are bounded as well and, 
by the Liouville theorem for the linear heat equation, 
all such solutions are constant. 

In nonlinear equations, a common way to 
prove the quasiconvergence of a solution is by
means of a Lyapunov functional. For equation \eqref{eq1}, the
following  energy functional is used frequently: 
\begin{equation}
  \label{eq:1}
  E(v):=\int_{-\infty}^\infty \big(\,\frac{v^2_x(x)
  }2-F(v(x))\big)\,dx,\qquad F(v):=\int_0^vf(s)\,ds.  
\end{equation}
Of course, for this functional to be defined along a solution, one
needs assumptions on $f$ and $u$; but when such assumptions are made,
it can be shown  that $t\mapsto E(u(\codt,t))$ is nonincreasing 
and consequently  $u$ is quasiconvergent (see, for
example, \cite{Feireisl:long-time} for results of this form). 

For solutions which are not assumed to be bounded in an integral norm,
the energy $E$ is not such a powerful tool.%
\footnote{Note, however, 
 that \cite{Gallay-S,Gallay-S2} made a good use of \eqref{eq:1} with the
 integral taken over the intervals $(-R,R)$, $R\gg 1$, instead of 
 $(-\infty,\infty)$. As proved in  \cite{Gallay-S},  the $\om$-limit
 set  of each bounded solution contains a steady state. A localized
 form of the energy functional is also used in the recent paper
 \cite{Risler:relaxation} in the proof of a quasiconvergence theorem for
 bistable solutions of gradient reaction-diffusion systems.  
} In fact, bounded solutions of nonlinear
equations \eqref{eq1} are not quasiconvergent in general. 
The existence of non-quasiconvergent solutions for some equations of
the form \eqref{eq1} was strongly indicated by results of
\cite{Eckmann-R}. It was later demonstrated by various examples
in \cite{P:examples, P:unbal}. Moreover, the results of 
\cite{P:examples, P:unbal} show that non-quasiconvergent 
bounded solutions occur quite ``frequently'' in \eqref{eq1}. They
exist whenever there is an interval $[a,b]$ on which $f$ is bistable:
$ f(a)$ and $f(b)$ are equal to zero, 
$ f'(a)$ and $f'(b)$ are negative, and there is $\ga\in (a,b)$ such
that $f$ is negative in $(a,\ga)$ and positive in $(\ga,b)$.
This is clearly a
 robust class of nonlinearities.

On the other hand, several classes of initial data $u_0$ have been
identified for which the solutions are quasiconvergent, if bounded. 
One such class is given by nonnegative  localized
data in the case $f(0)=0$ (cp. \cite{p-Ma:1d}).
We say that $u_0$ is \emph{localized} if
$u_0$ belongs to $C_0(\R)$---the space of all continuous
functions on $\R$ converging to $0$ at $x=\pm\infty$.
Under the stronger condition that $u_0\ge 0$ has compact support,
the solution has even been proved to bevcovnergent, if bounded;
see \cite{Du-M}.
Another suitable class of initial data
consists of front-like functions, by which  we mean functions
 $u_0\in C(\R)$ satisfying the relations
$a\le u_0\le b$, and having the limits $u_0(-\infty)=b$ and $u_0(\infty)=a$,
for some zeros  $a<b$ of $f$. The quasiconvergence in this case was
proved in \cite{P:prop-terr} (see 
\cite{P:quasiconv-overview} for a more detailed overview
of quasiconvergence and related results for equations of the form \eqref{eq1}). 
Note that in the case of localized initial data,
the quasiconvergence theorem is not valid without the sign
restriction;  examples of non-quasiconvergent solutions 
\cite{P:examples, P:unbal} do include some with sign-changing initial
data in  $C_0(\R)$. On the other hand, there are alternatives
to the sign condition that also guarantee
the quasiconvergence of solutions with localized initial data.
This is the subject of a sequel to the present paper.
A theorem in this direction, which moreover
applies in a more general setting of gradient reaction-diffusion systems,
can also be found in \cite{Risler:relaxation}.

In this paper,  we consider a class of initial data which includes in
particular all front-like initial data, but
without any sign restrictions like $a\le u_0\le b$.
Namely, we consider initial data $u_0$ in the space 
\begin{equation}\label{spacelimit}
 \cV:=\left\{ v\in C_b(\R):\textrm{ the limits }
v(-\infty),\,v(+\infty)\in \R \textrm{ exist}\right\}.
\end{equation}
Note that the property of having 
finite limits at  $\pm\infty$ is preserved by the
solutions of \eqref{eq1}, \eqref{ic1}: if $u_0(\pm \infty)$ exist, then 
$u(\pm\infty,t,u_0)$  exist for all $t\in (0,T(u_0))$ (these limits
vary with $t$ in general, see Lemma \ref{valueinfty} below 
for a more precise statement). This means that the space $\cV$ 
is an invariant space for \eqref{eq1}, just like the space 
$C_b(\R)$, or the space $C_0(\R)$ in the case  $f(0)=0$. Since 
$\cV$ is a closed subspace of $C_b(\R)$, it is
a Banach space when equipped with the supremum norm.  

Of course,  $\cV$ contains functions $u_0$ with
$u_0(-\infty)=u_0(\infty)$---in particular, it contains
$C_0(\R)$---so we do not have quasiconvergence of all bounded 
solutions with initial data $u_0\in \cV$. As it turns out, however, 
$u_0(-\infty)=u_0(\infty)$ is the only case when the quasiconvergence 
may fail to hold. This is a part of our main theorem, which we state
precisely after introducing some notation and terminology. 

Consider the ordinary differential equation for the steady states of
\eqref{eq1}:
\begin{equation}\label{steadyeq}
 u_{xx}+f(u)=0,\qquad x\in\R.
\end{equation}
The corresponding first-order system,
\begin{equation}\label{sys}
 u_x=v,\qquad v_x=-f(u),
\end{equation}
is a Hamiltonian system, which has only four types of bounded
orbits: equilibria, 
nonstationary periodic  orbits (or, closed orbits), 
homoclinic orbits, and heteroclinic orbits. We adopt 
the following common terminology concerning  steady states $\varphi$ of 
\eqref{eq1}. We say $\varphi$ is a \emph{ground state} of 
\eqref{steadyeq} if the orbit of
$(\varphi,\varphi')$ is a homoclinic orbit of \eqref{sys}; and 
$\varphi$ is a \emph{standing wave} of \eqref{eq1} 
 if the orbit of $(\varphi,\varphi')$ is a heteroclinic orbit 
of \eqref{sys}.

Under our standing hypothesis 
that $f$ is locally Lipschitz on $\R$, we have the following result. 

\begin{theorem}\label{mainthm}
 Assume that $u_0\in\cV$ and $u_0(-\infty)\ne u_0(\infty)$.
If the solution $u(\cdot,\cdot,u_0)$ of \emph{(\ref{eq1}), (\ref{ic1})} is
bounded, then it is quasiconvergent: $\omega(u_0)$ consists entirely of
steady states of  \emph{(\ref{eq1}).} More specifically, if 
$\varphi\in \omega(u_0)$, then 
it is  a constant steady state, or a ground state of
\emph{\eqref{steadyeq},} or a  
standing wave of  \emph{(\ref{eq1})}.
\end{theorem}
\begin{figure}[ht]\vspace{0.5cm} 
         \addtolength{\belowcaptionskip}{10pt}
         \addtolength{\abovecaptionskip}{-2.5cm}  
(a)   

\vspace{-5.5cm} 
\hspace{1cm}\includegraphics[scale=.45]{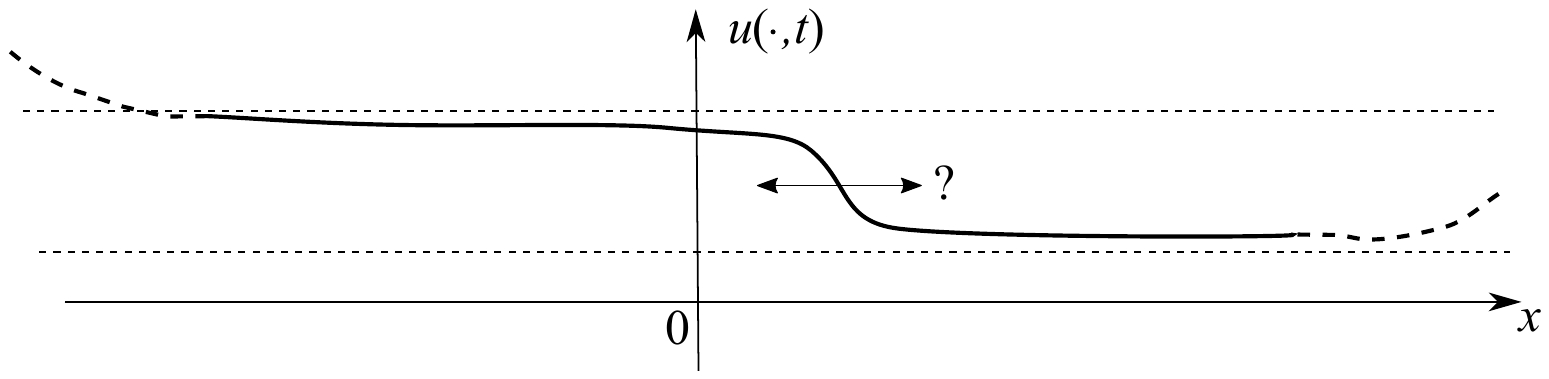} 
\vspace{-1.5cm} 

(b)

\vspace{-5.5cm} 
\hspace{1cm}\includegraphics[scale=.45]{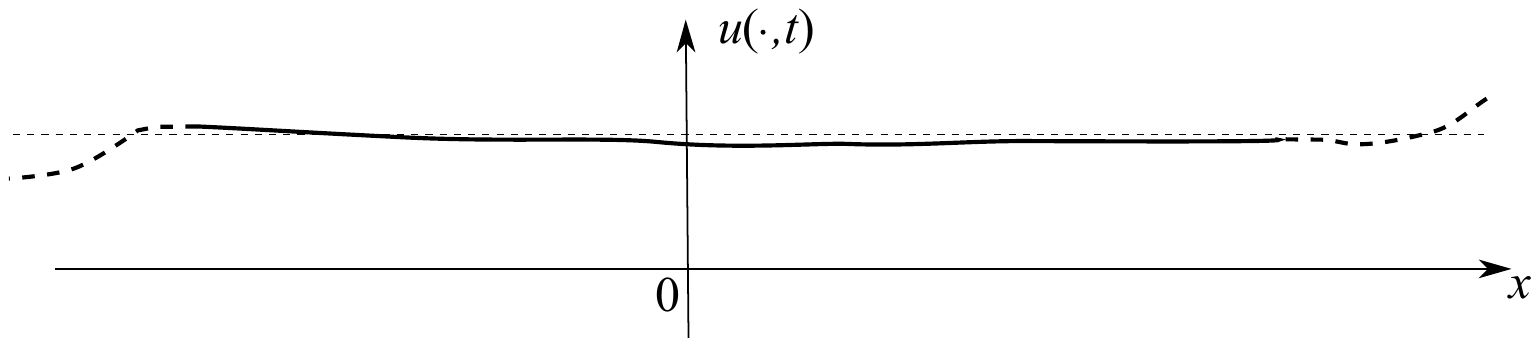} 
\vspace{-1.5cm} 

(c)

\vspace{-5.5cm} 
\hspace{1cm}\includegraphics[scale=.45]{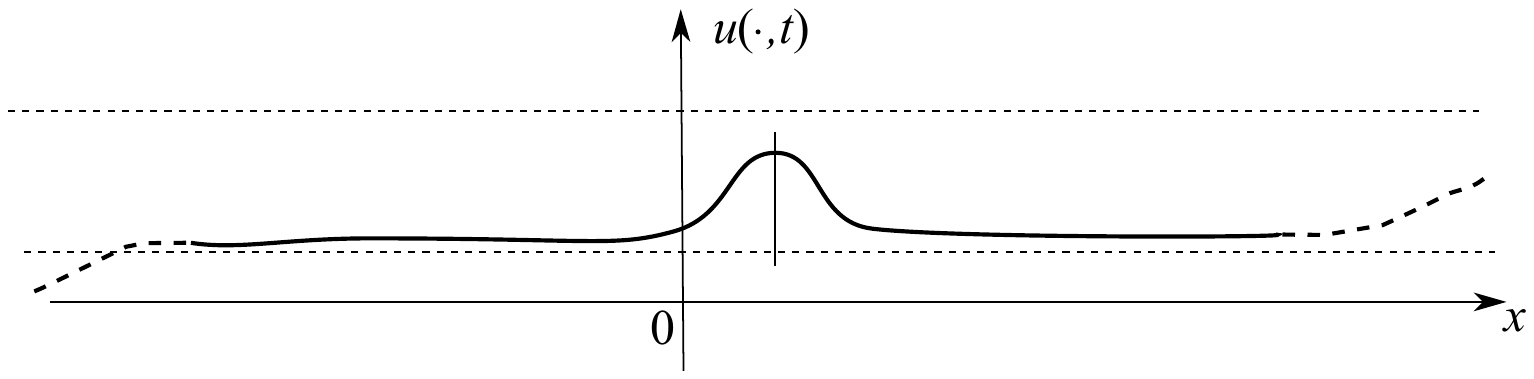}  
         \caption[Shape of $u$]{The large-time shape of $u(\cdot,t)$
           in bounded intervals. 
          If (M) holds, the shape may resemble a standing
          wave whose position may possibly shift slowly in either direction
          as time increases (Figure
           (a)). If (M) does not hold, the solution is convergent, its
           locally uniform limit being a constant (Figure (b)) or a
           ground state (Figure (c)).    
            \label{fig1}} 
         \end{figure}

\begin{Remark}
  \label{rmtothm}
{\rm 
  \begin{itemize}
  \item[(i)] Theorem \ref{mainthm} in particular shows that nonconstant
  periodic steady states are never elements of $\om(u_0)$ for $u_0\in
  \cV$. 
\item[(ii)] An even more precise description of  $\omega(u_0)$ will
    come out of the proof of the theorem. Namely, consider the
    following possibility:
    \begin{itemize}
    \item[(M)] $u$ is eventually
    monotone in space: given any $k\in \R$, one has $u_x(x,t)\ne 0$ for
    all $x\in (-k,k)$ if $t$ is sufficiently large. 
    \end{itemize}
If (M) holds, we will prove that $u$ is quasiconvergent and each
element of $\om(u_0)$ is a 
  constant steady state or a 
  standing wave of  (\ref{eq1}) (cp. Fig.~\ref{fig1}(a)).
  If (M) does not hold, we show
that $u$ is even convergent, with
$\om(u_0)=\{\varphi\}$, where $\vp$ is a 
constant steady state (Fig.~\ref{fig1}(b))  or a ground state
of {\eqref{steadyeq}} (Fig.~\ref{fig1}(c)).
Here, as always in this paper, the convergence  is with respect to the
topology of $L^\infty_{loc}(\R)$, that is, the locally uniform
convergence. The uniform convergence is not to be expected in general.
For example, if the limits of $u_0(x)$
at $x=-\infty$ and $x=\infty$ are two (distinct) zeros of $f$,
then the same is true, with the same limits, for
$u(x,t,u_0)$ at any positive time $t$. In this case, obviously,
$u(\cdot,t,u_0)$  cannot uniformly converge to a constant or a ground
state. 
\item[(iii)] Since the locally uniform convergence is used
  in the definition of  $\om(u_0)$,  Theorem \ref{mainthm}
  (and Figure \ref{fig1}) only give an information on the
  large-time shape of the solution $u(\cdot,t,u_0)$ restricted to
  (arbitrarily large) bounded spatial intervals. A different notion of the
  limit set would be more appropriate if one wanted to describe
  the global large-time shape of the solution, namely, 
  \begin{align*}
    \Omega(u_0) := \{ \vp\in C_b(\R):~&u(\cdot+x_n,t_n,u_0)\to\vp\\
      &\textrm{for some sequences }t_n\to\infty,\ ~ x_n\in\R\},
  \end{align*}
  where the convergence is in $L^\infty_{loc}(\R)$ again.
  Both $\om(u_0)$ and  $\Om(u_0)$ give relevant and interesting
  information on the large-time behavior the solution
  of $u(\codt,t,u_0)$. In this paper, we focus on the
  behavior in local spatial regions, which is a  specific
information captured by $\om(u_0)$ and not ``encoded'' in $\Om(u_0)$.
Thus, $\om(u_0)$ is the main object of our present study.
We remark, however, that the larger set $\Om(u_0)$   ``seldom''
consists of steady states: typically,
  one should expect traveling fronts of
  \eqref{eq1} to occur in $\Om(u_0)$ as well (see \cite{P:prop-terr}
  for a description of $\Om(u_0)$ in the case of front-like initial data
  $u_0$). 
\item[(iv)]  Clearly, the set of all functions $u_0\in \cV$ satisfying
$u_0(-\infty)\ne u_0(+\infty)$ is open and dense in $\cV$ 
(with the supremum norm). Thus, Theorem \ref{mainthm} has an
interesting additional feature in that it shows that 
quasiconvergence is generic in $\cV$: 
the solution of \eqref{eq1} is quasiconvergent, if bounded, for an
open and dense set of initial data in $\cV$. 
In contrast, using the constructions from \cite{P:examples},
 one can show that this genericity statement is not valid if one replaces
 $\cV$ with $C_b(\R)$ or $C_0(\R)$. 
\end{itemize}}
\end{Remark}

Theorem \ref{mainthm} is proved in Section \ref{proof}. Several
preliminary results concerning steady states, zero number, 
and $\om$-limit sets that are needed for 
the proof are recalled in Section \ref{prelims}. 

As Theorem \ref{mainthm} concerns bounded solutions only, 
modifying $f$ outside the range of the solution,
we may assume without loss of generality that $f$ satisfies the
following condition: 
\begin{equation}\label{coercivity}
 \text{there exists $\kappa>0$  such that for all $|u|>\kappa$
 one has $f(u)=\frac{u}{2}$.}
\end{equation}
This will be convenient in the next section. 

\section{Preliminaries}\label{prelims}

\subsection{Steady states and their trajectories in the phase plane}\label{stst}
In this subsection, we recall several technical 
results concerning the steady
states of (\ref{eq1}), or, solutions of (\ref{steadyeq}). The
first-order system \eqref{sys} corresponding to (\ref{steadyeq})
is Hamiltonian with respect to the energy
\begin{equation}\label{energy}
 H(u,v)=\frac{v^2}{2}+F(u),
\end{equation}
where $\displaystyle F(u)=\int_0^u f(s)ds.$ 
Thus, each orbit of (\ref{sys}) is contained in a level set of $H.$ 
The level sets  are symmetric with respect to the $v-$axis, and
our extra hypothesis (\ref{coercivity})  implies that they are all
bounded. Therefore, all orbits of  (\ref{sys}) are bounded 
and, as already mentioned in the introduction, there are only
four types of them: 
equilibria (all of which are on the $u-$axis), non-stationary periodic
orbits, homoclinic orbits (corresponding to ground states 
of (\ref{steadyeq})),
and heteroclinic orbits (corresponding to standing waves of (\ref{eq1})). 

Each non-stationary periodic orbit $\MO$ is symmetric about the $u-$axis and for some $p<q$ one has 
\begin{align}
 \MO\cap\{ (u,0):u\in\R\} & = \left\{(p,0),(q,0)\right\} \nonumber \\
 \MO\cap \left\{(u,v):v>0\right\} & = \left\{\lp u,\sqrt{2(F(p)-F(u))}\rp:u\in(p,q)\right\}. \label{periodicorbits}
\end{align}

The following result  of \cite{p-Ma:1d} gives a
 description of the phase plane portraits of (\ref{sys}) with all the
 periodic orbits removed.  Let 
\begin{align}
 \ME & := \{ (a,0):f(a)=0\} \textrm{ (the set of all equilibria of (\ref{sys}))}, \nonumber \\
 \MP_0 & :=\{(a,b)\in\R^2: (a,b)\textrm{ lies on a non-stationary periodic orbit of (\ref{sys})}\} \nonumber \\
 \MP & := \MP_0\cup\ME. \nonumber
\end{align}

\begin{lem}\label{MatPolLemma}
 \emph{\cite[Lemma 3.1]{p-Ma:1d}} The following two statements are valid.
 \begin{enumerate}
  \item[(i)] Let $\Sigma$ be a connected component of $\R^2\setminus\MP_0.$ Then $\Sigma$ is a compact set contained in a level set of the Hamiltonian $H$ and one has 
  \begin{equation*}
   \Sigma = \left\{(u,v)\in\R^2:u\in J,\ v=\pm\sqrt{2(c-F(u))}\right\}
  \end{equation*}
where $c$ is the value of $H$ on $\Sigma$ and $J=[p,q]$ for some $p,q\in\R$ with $p\leq q.$  Moreover, if $(u,0)\in\Sigma$ and $p<u<q,$ then $(u,0)$ is an 
equilibrium. The point $(p,0)$ is an equilibrium or it lies on a homoclinic orbit; the same is true for the point $(q,0).$
\item[(ii)] Each connected component of the set $\R^2\setminus\MP$ consists of a single orbit of \emph{(\ref{sys})}, either a homoclinic orbit or a heteroclinic orbit.
 \end{enumerate}
 \end{lem}
The following lemma is a simple consequence 
of the continuity of the solutions of  (\ref{sys}) 
with respect to the initial conditions. 
\begin{lem}\label{lemperiods}
 Let $\MO_n$, $n=1,2,\dots$ be a sequence of non-stationary periodic
 orbits of \emph{(\ref{sys})} with the  minimal periods $\rho_n,$
 $n=1,2,\dots$, respectively.  Suppose that for some bounded set 
$K\subset \R^2\setminus\MP$ one has 
 $\displaystyle\textrm{dist}\lp\MO_n,K\rp\underset{n\to\infty}{\longrightarrow}0.$ Then $\displaystyle \rho_n\underset{n\to\infty}{\longrightarrow}\infty.$
\end{lem}

If $\vp$ is a $C^1$ bounded function on $\R,$ we let
\begin{equation*}
 \tau(\vp):=\left\{ \lp \vp(x),\vp_x(x)\rp:x\in\R\right\} 
\end{equation*}
and refer to this set as 
the \textit{spatial trajectory (or orbit)} of $\vp.$ 
If $\vp$ is a solution of (\ref{steadyeq}), then $\tau(\vp)$ is the usual  
orbit of the solution $(\varphi,\varphi_x)$ of (\ref{sys}).

\subsection{Invariance of the $\om$-limit set}\label{invariance}
Recall that the $\omega-$limit set of a bounded  solution $u$ of
\eqref{eq1}, denoted by  
$\omega(u)$, or $\omega(u_0)$  if the initial value of $u$ is given,
is defined as in \eqref{defomega}, with the convergence in 
$L^\infty_{loc}(\R)$.  By standard parabolic estimates the trajectory
$\{ u(\cdot,t),\ t\geq1\}$  of $u$  
is relatively compact in $L^\infty_{loc}(\R).$ This implies that
$\omega(u)$ is nonempty, compact, and connected in (the metric space) 
$L^\infty_{loc}(\R)$
and it attracts the solution in the following sense:
\begin{equation}
  \label{eq:11}
  \textrm{dist}_{L^\infty_{loc}(\R)}\lp u(\cdot,t),\omega(u)\rp\underset{t\to\infty}{\longrightarrow}0.
\end{equation}
It is also a standard observation that if $\vp\in\omega(u),$ there exists 
an entire solution $U(x,t)$ of (\ref{eq1})
 such that 
\begin{equation}\label{entiresol}
 U(\cdot,0)=\vp,\qquad U(\cdot,t)\in\omega(u)\quad (t\in\R).
\end{equation}
Here, \textit{an entire solution} of (\ref{eq1}) refers to a
solution defined for all $x\in\R$, $t\in\R.$
Let us briefly recall how such an entire solution $U$ is found.  
By  parabolic regularity estimates,  $u_t,u_x,u_{xx}$
are bounded on $\R\times[1,\infty)$ and are globally $\a-$H\"older for
any $\a\in(0,1).$ 
If $u(\cdot,t_n)\underset{n\to\infty}{\longrightarrow}\vp$ in
$L^\infty_{loc}(\R)$  for 
some $t_n\to\infty,$ we consider the sequence $u_n(x,t):=u(x,t+t_n)$,
$n=1,2\dots$.   
Passing to a subsequence if necessary, 
we have  $u_n\to U$  in $C^1_{loc}(\R^2)$ 
for some function $U$; this function $U$ is  then 
easily shown to be an entire solution of (\ref{eq1}).
By definition, $U$ satisfies \eqref{entiresol}. Note that 
the entire solution $U$ is determined uniquely by $\varphi$; this
follows from the uniqueness and backward uniqueness for the Cauchy
problem  \eqref{eq1}, \eqref{ic1}. 

Using similar  compactness arguments, one shows easily that
 $\omega(u)$ is connected in $C_{loc}^1(\R).$ 
Hence, the set 
$$
\left\{ (\vp(x),\vp_x(x)):\vp\in\omega(u),x\in\R\right\} = \underset{\vp\in\omega(u)}{\cup}\tau(\vp)
$$
is connected in $\R^2.$ Also, obviously, 
$\tau(\vp)$ is  connected in $\R^2$ for all $\vp\in\omega(u).$

We will also use the following result 
(see  \cite[Lemma 4.3]{p-Ma:1d}  or
 \cite[Lemma 6.10]{P:prop-terr} for a proof).
\begin{lem}\label{translatedorbit}
Let $u$ be a bounded solution of \eqref{eq1}. If
 $\vp\in\omega(u),$ $\psi$ is a solution of (\ref{steadyeq}),
and  $\tau(\vp)\subset\tau(\psi),$ then $\vp$ is a shift of $\psi.$
\end{lem}

\subsection{Zero number for linear parabolic equations}
In this subsection, we consider solutions of a linear parabolic equation 
\begin{equation}\label{eqlin}
 v_t=v_{xx}+c(x,t)v,\qquad x\in\R,\ t\in\lp s,T\rp,
\end{equation}
where $-\infty\leq s<T\leq \infty$ and $c$ is a bounded measurable
function. For an interval $I=(a,b),$ with $-\infty\leq a < b\leq
\infty,$ we denote by $z_I(v(\cdot,t))$ 
the number, possibly infinite, of zeros $x\in I$ of the function
$x\mapsto v(x,t).$ If $I=\R$, we usually omit the subscript $\R$: 
$$
z(v(\cdot,t)):=z_\R(v(\cdot,t)).
$$
The following intersection-comparison principle holds
 \cite{Angenent:zero,Chen:strong}.
\begin{lem}\label{lemzero}
 Let $v\not\equiv 0$ be a solution of \emph{(\ref{eqlin})} and
 $I=(a,b),$ with $-\infty\leq a < b\leq \infty.$ Assume that the
 following conditions are satisfied: 
  \begin{itemize}
  \item if $b<\infty,$ then $v(b,t)\neq0$ for all $t\in\lp s,T\rp,$
  \item if $a>-\infty,$ then $v(a,t)\neq0$ for all $t\in\lp s,T\rp.$
 \end{itemize}
Then the following statements hold true.
\begin{enumerate}
 \item[(i)] For each $t\in\lp s,T\rp,$ all zeros of $v(\cdot,t)$ are
   isolated. In particular, if $I$ is bounded, then
   $z_I(v(\cdot,t))<\infty$ for all $t\in\lp s,T\rp.$ 
 \item[(ii)] The function $t\mapsto z_I(v(\cdot,t))$ is monotone
   non-increasing on $(s,T)$ with values in
   $\N\cup\{0\}\cup\{\infty\}.$ 
 \item[(iii)] If for some $t_0\in(s,T)$ the function $v(\cdot,t_0)$
   has a multiple zero in $I$ and $z_I(v(\cdot,t_0))<\infty,$ then for
   any $t_1,t_2\in(s,T)$ with $t_1<t_0<t_2,$ one has 
 \begin{equation}\label{zerodrop}
   z_I(v(\cdot,t_1))>z_I(v(\cdot,t_0))\ge z_I(v(\cdot,t_2)).
 \end{equation}
 \end{enumerate}
\end{lem}
If (\ref{zerodrop}) holds,  we say that $z_I(v(\cdot,t))$ drops in the
interval $(t_1,t_2).$ 

\begin{remark}\label{convzero}{\rm
 It is clear that if the assumptions of Lemma \ref{lemzero} are
 satisfied and for some $t_0\in(s,T)$ one has
 $z_I(v(\cdot,t_0))<\infty,$ then  
 $z_I(v(\cdot,t))$ can drop at most finitely many times in $(t_0,T)$;
 and if it is constant on $(t_0,T),$ then $v(\cdot,t)$ has only simple
 zeros in $I$  
 for all $t\in (t_0,T).$ In particular, if $T=\infty,$ there exists
 $t_1<\infty$ such that $t\mapsto z_I(v(\cdot,t))$ is constant on
 $(t_1,\infty)$ and for $t>t_1$ all zeros  of $v(\cdot,t)$ 
 are simple.	
}
\end{remark}

Using the previous remark and the implicit function 
theorem, we obtain the following corollary.
\begin{cor}\label{zeroIFT}
 Assume that the assumptions of Lemma \ref{lemzero} are satisfied and that the function $t\mapsto z_I(v(\cdot,t))$ is constant on $(s,T).$
 If for some $(x_0,t_0)\in I\times(s,T)$ one has  $v(x_0,t_0)=0,$ then
 there exists a $C^1$- function $t\mapsto\eta(t)$ defined for
 $t\in(s,T)$ such that  $\eta(t_0)=x_0$ 
and $v(\eta(t),t)=0$ for all $t\in(s,T).$ 
\end{cor}

We will also need the following robustness lemma 
(see \cite[Lemma 2.6]{Du-M}).
\begin{lem}\label{robustnesszero}
 Let $w_n(x,t)$ be a sequence of functions converging to $w(x,t)$ in
 $\displaystyle C^1\lp I\times(s,T)\rp$ where $I$ is an open
interval. Assume that $w(x,t)$ 
 solves a linear equation (\ref{eqlin}), $w\not\equiv0$, and  $w(\cdot,t_0)$
 has a multiple zero  $x_0\in I$ for some $t_0\in(s,T)$.
 Then there exist sequences $x_n\to x_0$, $t_n\to t_0$
 such that for all  sufficiently large 
$n$ the function $w_n(\cdot,t_n)$ has a multiple zero at $x_n$. 
\end{lem}

In the next section we frequently use the following
standard facts, often
without notice. If $u$, $\bar u$ are bounded solutions of 
the nonlinear equation \eqref{eq1} with a Lipschitz nonlinearity,
then their  difference  $v=u-\bar u$ 
satisfies a linear equation \eqref{eqlin} with some bounded measurable
function $c$.  Similarly, $v=u_x$ and $v=u_t$  are 
 solutions of such a linear equation.

\section{Proof of the main result}\label{proof}
Throughout this section we assume the hypotheses of Theorem
\ref{mainthm} to be satisfied:  
$u_0\in \cV$ (cp. \eqref{spacelimit}) and 
\begin{equation}
  \label{hyponu0}
  \al:=u_0(-\infty)\ne \be:= u_0(+\infty).
\end{equation}
Further, we assume that the solution 
 $u(x,t)$ of
(\ref{eq1})-(\ref{ic1}) is bounded.

We denote 
\begin{equation}\label{limits-t}
 \theta_-(t):=\lim_{x\to-\infty}u(x,t),\qquad \theta_+(t):=\lim_{x\to\infty}u(x,t).
\end{equation}
These limits exist according to the following lemma (the proof can be
found in \cite[Theorem
  5.5.2]{Volpert-V-V}, for example). 
\begin{lem}\label{valueinfty}
 The limits $\theta_-(t),\theta_+(t)$ exist for all  $t> 0$ 
and are solutions of the following initial-value problems:
 \begin{equation}\label{valueinftyeq}
  \dot{\theta}_\pm=f(\theta_\pm),\qquad \theta_-(0)=\a,\ \theta_+(0)=\b.
 \end{equation}
\end{lem}
\subsection{The reflection principle and stabilization of the critical
  points}\label{refl-sec}
We employ the reflection-invariance of equation \eqref{eq1} 
in a way quite common in studies of spatially
homogeneous parabolic equations. 
For any $\la\in\R,$ consider the function
$V_\la u$ defined by
\begin{equation}\label{reflexion}
 V_\la u(x,t)=u(2\la-x,t)-u(x,t),\quad x\in\R,\,t\ge 0.
\end{equation}
Being the difference of two solutions of \eqref{eq1}, $V_\la u$
is a solution of the linear equation (\ref{eqlin}) for some bounded function
$c$. 

We apply zero-number results to the functions $V_\la u$, $\la\in
\R$. First observe that for any $\la\in\R$,  hypothesis (\ref{hyponu0})
and Lemma \ref{valueinfty} imply that for $t\ge 0$ the function
$V_\la u(x,t)$ has the limits as $x\to\pm\infty$ given by
$\pm(\theta^+(t)-\theta^-(t))$, and these limits are both nonzero for
all sufficiently small $t>0$. Therefore, by Lemma \ref{lemzero},
$z(V_\la u(\cdot,t))$ is finite for all $t>0.$ 
By Remark \ref{convzero}, there  is 
 $T=T(\la)$ such that the function 
$t\mapsto z(V_\la u(\cdot,t))$ is constant on $\lp T(\la),\infty\rp$ and, for all $t>T(\la),$ all zeros 
of $V_\la u(\cdot,t)$ are simple.
In particular, since $x=\la$ is always a zero of $V_\la u$ 
by the definition of $V_\la u$,
we  have
\begin{equation}\label{uxneqzero}
-2u_x(\la,t)=\partial_x V_\la u(x,t)\vert_{x=\la}\neq0\quad (t>T(\la)).
\end{equation}
We use this to prove the following result (a similar theorem for
solutions periodic in space can be found in \cite{Chen-M:JDE}).

\begin{prop}\label{loczero}
 For any open bounded interval $I\subset\R,$ there exist
 $T_1=T_1(I)>0$ and an integer $N\ge 0$ such that, for all $t>T_1,$  
the function
  $u_x(\cdot,t)$ has exactly $N$ zeros in $I$, all of them simple.
Moreover, if $N>0$ and  $\eta_1(t)<\dots \eta_N(t)$ denote the zeros
of  $u_x(\cdot,t)$ in $I$ for $t>T_1$, then the functions 
$\eta_i(t),$ $i=1,\dots, N,$ are of class $C^1$, and for some 
$x_i^\infty\in\overline{I}$ one has 
\begin{equation}\label{loczeroeq}
\eta_i(t)\underset{t\to\infty}{\longrightarrow}x_i^\infty\quad(i=1,\dots,N). 
 \end{equation}
\end{prop}

\begin{proof}
Let $I=(a,b)$ with real $a$ and $b$ be any open bounded interval. 
Applying  (\ref{uxneqzero}) to $\la\in \{a,b\}$ and setting 
$T_0:=\max\lp T(a),T(b)\rp$, we obtain
\begin{equation*}
 u_x(a,t)\neq0\neq u_x(b,t)\quad (t>T_0). 
\end{equation*}
The function $u_x(x,t)$ is a solution of a linear equation 
(\ref{eqlin}). Hence, by Lemma \ref{lemzero},
there exist $T_1\geq T_0$ and $N\geq0$ such that  for all 
$t>T_1$ the function
  $u_x(\cdot,t)$ has exactly $N$ zeros in $I$, all of them simple.
If $N=0$, the proof of 
 Proposition \ref{loczero} is finished. 

Assume that $N\geq1.$ Let $\eta_1(t)<\dots <\eta_N(t)$
the zeros of  $u_x(\cdot,t)$ in $I$. The simplicity of these zeros and
the implicit function theorem imply that the functions 
$\eta_i(t),$ $i=1,\dots, N,$ are of class $C^1$. 
It remains to show that these functions are all convergent. 
Assume for a contradiction that for some $i\in \{1,\dots,N\}$
the function
 $\eta_i(t)$ is not convergent as $t\to\infty.$
Then it admits at least two 
accumulation points $x_i^\infty\neq\tilde{x}_i^\infty.$ 
Consequently, for $\la:=(x_i^\infty+\tilde{x}_i^\infty)/{2}$ there is a
sequence  $t_n\to\infty$ such that 
$\eta_i(t_n)=\la$ for all $n$, which is a
contradiction to (\ref{uxneqzero}).
\end{proof}

\begin{Remark}
  \label{rm:equal-lim}
  {\rm The above reflection argument is the only  place in this paper
    where need the hypothesis that the limits $\al=u_0(-\infty)$
    and $\be=u_0(\infty)$ are distinct. More specifically, we only
    need the hypothesis to ensure
    that the zero number $z(V_\la(\cdot,t)$ is finite for
    any $t>0$ and for any $\la\in\R$. If $\al$ and $\be$ are equal and
    the solution 
    $u(\cdot,\cdot,u_0)$ happens to have the previous property,
    then the results of
    this paper, and in particular Theorem \ref{mainthm}, are valid for
    this solution.  
    }
\end{Remark}

We complement the  result of Proposition \ref{loczero} with the
following useful information.   

\begin{lemma}
  \label{le-conv0}
Under the notation of Proposition \ref{loczero}, if $N>0$, then 
for any $\la\in \{x_i^\infty: i=1,\dots,N\}$ one has 
\begin{equation}
  \label{eq:3}
  V_{\la}u(\codt,t)\underset{t\to\infty}{\longrightarrow}0\ \text{ 
in $C^1_{loc}(\R)$.} 
\end{equation}
\end{lemma}

\begin{proof}
It is sufficient to prove the following statement. 
Given any sequence $t_n\to\infty$, one can pass 
to  a subsequence such that \eqref{eq:3} holds with $t$ replaced
by $t_n$. 

This will be shown using an entire solution $U$, as constructed in 
Section \ref{invariance}. Passing to a subsequence of $\{t_n\}$, we
may assume that $u(\cdot,\cdot+t_n)\to U$  in $C^1_{loc}(\R^2)$,
where $U$ is an entire solution of \eqref{eq1}. Then also
$V_\la u(\cdot,\cdot+t_n)\to V_\la U$  in $C^1_{loc}(\R^2)$, for any
$\la$. If now $\la\in \{x_i^\infty: i=1,\dots,N\}$, then 
\begin{equation*}
  \partial_x V_\la u(x,t)=-2u_x(\la,t)\underset{t\to\infty}{\longrightarrow}0.
\end{equation*}
It follows that 
\begin{equation*}
  \partial_x V_\la U(\la,t)=0 \quad (t\in\R).
\end{equation*}
Since we  also have $V_\la U(\la,t)=0$ (cp. \eqref{reflexion}), $x=\la$
is a multiple zero of $V_\la U(\cdot,t)$ for all $t\in\R$. Using Lemma
\ref{lemzero}, one shows easily that this is only possible if 
$V_\la U\equiv 0$. This in particular yields the desired conclusion:
\begin{equation*}
  V_{\la}u(\codt,t_n){\to}0\ \text{ 
in $C^1_{loc}(\R)$.} \qedhere
\end{equation*}
\end{proof}

Take now the intervals  $I_k:=\lp-k,k\rp,$ $k=1,2,\dots$, 
 and let $N_k$ be the number of zeros of $u_x(\cdot,t)$ in $I_k$ for
 $t>T_1(I_k)$, where $T_1$ is as in Proposition \ref{loczero}.
We distinguish the following mutually exclusive cases.
\begin{itemize}
 \item[(C1)] There is $k_0$ such that
 $N_k=0$ for $k=k_0,k_0+1,\dots$. 
 \item[(C2)] There is $k_0$ such that
 $N_k=1$ for $k=k_0,k_0+1,\dots$.
 \item[(C3)] There is $k_0$ such that
 $N_k\ge 2$ for $k=k_0,k_0+1,\dots$.
\end{itemize}
According to Proposition \ref{loczero}, 
(C1) means that each bounded interval is free of critical points of 
$u(\cdot,t)$ for $t$ large enough. In the case (C2),  
$u(\cdot,t)$ has exactly one critical point $\eta(t)$ such that 
$\eta(t)$ has a finite limit as $t\to\infty$; moreover, in any bounded
interval, $u(\cdot,t)$ has no critical points different from
$\eta_1(t)$ for $t$
large enough. In the case (C3), there are more than one critical points
of $u(\cdot,t)$ with finite limit as $t\to\infty$. 

We give the proof of Theorem \ref{mainthm} in each of these cases
separately. 

\subsection{Case (C1): no limit critical point}
We consider case (C1) here. Clearly, (C1) implies that
$u_x(\cdot,t)$ is of one sign in $I_k$ for large $t$ and this sign is
independent of $t$. Without loss of generality, replacing $u(x,t)$ by
$u(-x,t)$ if necessary, we assume that for all $k$ one has 
\begin{equation}\label{uxnegative}
 u_x(x,t)<0\quad (x\in(-k,k),\ t>T(I_k)).
\end{equation}
In this situation, we have the following result concerning the
$\om$-limit set $\om(u)$:

\begin{lem}\label{nointersection}
Let  $\psi$ be any nonconstant periodic solution of {\rm (\ref{steadyeq})}.
Then 
 $$
 \tau(\vp)\cap\tau(\psi)=\emptyset \quad (\varphi\in \om(u)).
 $$
\end{lem}
\begin{proof}
 We go by contradiction. Assume that there is $\varphi\in \om(u)$ 
 such that  $\tau(\vp)\cap\tau(\psi)\neq \emptyset.$ 
 This means that, possibly after  replacing $\psi$ by 
a translation,  there exists $x_0\in\R$ such that  
 $$
 \psi(x_0)=\vp(x_0),\qquad \psi'(x_0)=\vp'(x_0).
 $$
 To simplify the notation, we will  assume without 
loss of generality that $x_0=0$ (this can be achieved by a translation
in the original equation, with no effect on the validity of the conclusion).

 Let $U(x,t)$ be an entire solution as in 
(\ref{entiresol}). There exists a sequence $t_n\to\infty$ such that 
 $$
 u(\cdot,\cdot+t_n)\underset{n\to\infty}{\longrightarrow}U\textrm{ in }C^1_{loc}(\R^2).
 $$
Then the sequence $\displaystyle w_n:=u(\cdot,\cdot +t_n)-\psi$
converges in $C^1_{loc}(\R^2)$ to the function $w(x,t):=U(x,t)-\psi(x).$
The function $w$ solves a linear equation (\ref{eqlin}) and
$w(\cdot,0)$ has a multiple zero at $x=0.$  Also, 
$w(\cdot,0)=\varphi -\psi\not\equiv 0$, for $\psi$ is nonconstant
periodic, whereas for $\vp$ we have, as a direct
consequence of \eqref{uxnegative}, that $\vp'\leq0.$
Applying Lemma \ref{robustnesszero}, 
we obtain that  there exist sequences $x_n\to0,$ $\d_n\to0,$ such that
$w_n(\cdot,\d_n)$ has a multiple zero at $x=x_n.$ Consequently,
with  $\tilde{t}_n:=t_n+\de_n$ we have 
$\tilde{t}_n \to\infty$, $x_n\to0$,  and 
 \begin{equation}\label{doublezeronocritical}
  u(\cdot,\tilde{t}_n)-\psi \textrm{ has a multiple zero at }x={x_n}.
 \end{equation}
We show that (\ref{doublezeronocritical}) contradicts
(\ref{uxnegative}). 

Since  $\psi$ is periodic, there are
$\rho_-<0<\rho_+$ such that  $\psi(\rho_-)=\min\psi,$ $\psi(\rho_+)=\max\psi,$ and ${x_n}\in(\rho_-,\rho_+),$ for all $n.$
Consider the function 
$$
z(t):=z_{(\rho_-,\rho_+)}\lp u(\cdot,t)-\psi\rp.
$$
By (\ref{uxnegative}), there is $T_0>0$ such that for all $t>T_0$  we
have  $u_x(\cdot,t)<0$ on $(\rho_--1,\rho_++1).$ Thus, if
$z(t)>0,$ then 
\begin{equation}
  \label{eq:2}
  u(\rho_-,t)>\psi(\rho_-)\text{ and }u(\rho_+,t)<\psi(\rho_+).
\end{equation}
Therefore, Lemma \ref{lemzero} implies that 
\begin{equation}\label{zdecreasing}
 t\mapsto z(t) \textrm{ is nonincreasing on any subinterval of }\{t:t>T_0 \textrm{ and }z(t)>0\}.
\end{equation}
Assertion (\ref{doublezeronocritical}) implies that for some
 $T_1>T_0$ one has $z(T_1)>0.$ We consider two complementary cases:

\textit{Case 1:} $z(t)>0,$ for all $t>T_1.$ In this case, 
\eqref{eq:2} shows that Lemma \ref{lemzero} and Remark \ref{convzero}
apply to $v=u-\psi$. Therefore, for all large $t$, 
$u(\cdot,t)-\psi$ has only simple zeros in $(\rho_-,\rho_+)$. This is
a contradiction to  (\ref{doublezeronocritical}).

\textit{Case 2:} there exists $T_2>T_1$ such that 
$z(T_2)=0.$ Pick  large enough $n_0$ such that $\tilde t_{n_0}>T_2$ and set 
$$\displaystyle T_3:=\sup\{t\in [T_2,\tilde t_{n_0}):
z(t)=0\}.
$$
From (\ref{doublezeronocritical}) we know that $T_3<\tilde t_{n_0}.$ The
definition of $T_3$ and the monotonicity of $x\mapsto u(x,T_3)$ 
(cp. \ref{uxnegative})
implies that either
$u(\rho_-,T_3)=\psi(\rho_-)$ or  
$u(\rho_+,T_3)=\psi(\rho_+).$ 
We only consider the  first possibility, 
the other being similar. 
In addition to the relation
$u(\rho_-,T_3)=\psi(\rho_-)$, we have,
by (\ref{uxnegative}) and the definition of $\rho_-$, 
the relations $u_x(\rho_-,T_3)<0=\psi_x(\rho_-)$.
The implicit function theorem therefore implies 
that there exists a continuous
function $t\mapsto\eta(t)$ defined on a neighborhood of $T_3$ such
that on a neighborhood  of the point
of $(\rho_-,T_3)\in \R^2$  one has
$u(x,t)=\psi(x)$ if and only if $x=\eta(t).$
Using this,  (\ref{uxnegative}) and the continuity of $u,$ 
we find  $\e>0$ such that  
$u(x,t)=\psi(x)$ holds with $x\in(\rho_--\e,\rho_+]$ and $t\in[T_3,T_3+\e)$
only if $x=\eta(t)$. 
As a result, for $t>T_3$ close enough to $T_3,$ 
we have $z(t)\le 1$. The definition of 
$T_3$ implies that  $z(t)>0$ for all $t\in (T_3,t_{n_0}]$. From
this obtain that, first, 
$z(t)= 1$ for all $t>T_3$, $t\approx
T_3$, and, second, (\ref{zdecreasing}) applies to the interval  
$(T_3,t_{n_0}]$. Consequently, 
$z(t)= 1$ for all  $t\in (T_3,t_{n_0}]$. Using this,
\eqref{eq:2}, and (\ref{doublezeronocritical}) with $n=n_0$, 
we now obtain a contradiction to Lemma \ref{lemzero}(iii) (take
$t_0=t_{n_0}$ in \eqref{zerodrop}).   
\end{proof}

\begin{proof}[Proof of theorem \ref{mainthm} in the case {\rm (C1)}] 
 Assuming (\ref{uxnegative}), we show that any $\varphi\in \om(u)$ is
 either a constant steady state or a standing wave of \eqref{eq1}.   

By  (\ref{uxnegative}),   $\vp_x\leq0$. Also, from  
Lemma \ref{nointersection} we know that, in the notation of 
Lemma \ref{MatPolLemma}, $\tau(\vp)\subset \R^2\setminus\MP_0$.

If $\vp_x\equiv0$, then $\tau(\vp)$ consists of the 
single point $(\varphi(0),0)$.
According to Lemma \ref{MatPolLemma}, this point
 is an equilibrium of \eqref{sys} or 
is contained in $\tau(\psi)$, where $\psi$ is a ground state
solution of (\ref{steadyeq}). In the later case, 
from Lemma \ref{translatedorbit} we obtain 
$\vp$ is a shift of $\psi,$ which is impossible as $\vp_x\equiv 0$. 
Thus, in this case, $\vp$ is a constant steady state.  

Assume now that $\vp_x\not\equiv 0.$ We first show that
 $\vp_x<0$ on $\R.$ Indeed, let $U(x,t)$ be the entire 
solution of (\ref{eq1}) as in (\ref{entiresol}): 
 $U(\cdot,0)=\vp$ and $U(\cdot,t)\in \om(u)$ for all $t\in\R$.  
The latter implies that $U_x (\cdot,t)\le 0$ for all $t$ and the
former gives  $U_x (\cdot,0)\not\equiv 0$. Therefore, applying the maximum
principle to $U_x$, we obtain $\vp_x=U_x (\cdot,0)< 0$, as
desired. In particular, $\tau(\vp)$ does not intersect the $u-$axis in
the $u-v$ plane. This and Lemma \ref{MatPolLemma}(ii) imply
that  $\tau(\vp)\subset\tau(\psi)$, where $\tau(\psi)$
is either a heteroclinic orbit of \eqref{sys} or a homoclinic orbit
of \eqref{sys}. By 
Lemma \ref{translatedorbit}, $\vp$ is a shift of $\psi$.
This and the relation $\vp_x<0$ in particular imply that $\tau(\psi)$
cannot be a homoclinic orbit.   Thus, 
$\tau(\psi)$ is a heteroclinic orbit of \eqref{sys}, meaning that
$\psi$ is a standing wave of \eqref{eq1}; and $\phi$, being a shift of
$\psi$, is a standing wave itself.
\end{proof}

\subsection{Case (C2): a unique limit critical point}
In the case (C2), there exists a $C^1$
function $t\mapsto \eta(t)$ defined on an interval $(T_0,\infty)$ with
with  $\eta(t)\underset{t\to\infty}{\longrightarrow}\eta^\infty\in\R$ 
and with the following property. For each $k\in\{k_0,k_0+1,\dots\}$ 
there is   $T(I_k)$ such that
\begin{equation}\label{uxuniquezero}
 \{(x,t):u_x(x,t)=0, x\in(-k,k),\ t>T(I_k)\}=\{(\eta(t),t):t>T(I_k)\}.
\end{equation}
Without loss of generality, using a shift if necessary,  we will
further assume that $\eta^\infty=0.$ 

By Proposition \ref{loczero}, $x=\eta(t)$
is a simple zero of $u_x(\cdot,t)$, so $u(\cdot,t)$ has  a strict
 local minimum or a strict local maximum at   $\eta(t)$.
We only consider the latter, the former  is 
analogous. Thus, we henceforth assume that 
\begin{equation}\label{maxu}
 u\lp \eta(t),t\rp = \underset{x\in(-k,k)}{\max}u(\cdot,t),\qquad t>T(I_k).
\end{equation}

From \eqref{uxuniquezero}, \eqref{maxu},  
and Lemma \ref{le-conv0}, we obtain that each
 $\vp\in\omega(u)$ has the following properties:
\begin{equation}\label{vpeven}
 \max_{\R}\vp = \vp(0), \quad \vp_x(x)\le 0\quad(x>0),
\quad \vp(-x)=\vp(x)\quad(x\in\R).
\end{equation}
Moreover,  for each $\vp\in\omega(u)$
\begin{equation}\label{monotonicity2}
 \textrm{either }\vp_x\equiv0 \textrm{ or }\vp_x(x)<0 \textrm{ for all } x>0.
\end{equation}
To prove this, let $U$ 
be the entire solution of (\ref{eq1}) with $U(\cdot,0)=\vp$ 
and $U(\cdot,t)\subset\omega(u)$  for all $t\in \R$.
Then, by \eqref{vpeven}, for all $t\in \R$ we have 
$U_x(0,t)=0$ and 
$U_x(\cdot,t)\leq0$
on $[0,\infty)$. Since the function
$U_x$ satisfies a linear equation (\ref{eqlin}),
the maximum principle implies that 
$\vp_x=U_x(\cdot,0)$ is either identical to zero 
or strictly negative on $(0,\infty)$.

Our goal now is to prove the following result.

\begin{prop}\label{propcontinuum} For some $\gamma\in\R$, one has
 $
 \{\vp(0):\vp\in\omega(u)\}=\{\gamma\}.
 $
\end{prop}
Assuming that this true, we now complete the proof
of Theorem \ref{mainthm} in the case (C2); we show that $u$ is even
convergent in this case. Then we give the proof of
Proposition \ref{propcontinuum}.

\begin{proof}[Proof of Theorem \ref{mainthm} in the case {\rm (C2)}]
  We prove that $\omega(u)$ consists of a single element,
either a ground state or a constant steady state.

Given any $\vp\in \omega(u).$ Let $U$ be the 
entire solution of (\ref{eq1}) with $U(\cdot,0)=\vp,$ 
and $U(\cdot,t)\in\omega(u)$ for all $t\in\R.$
By Proposition \ref{propcontinuum}, 
$U(0,t)=\gamma$ for all $t\in\R$. Therefore, $U_t(0,\cdot)\equiv 0$.
Moreover, by (\ref{vpeven}), 
we have $U_x(0,t)=0$ for all $t$, thus $U_{xt}(0,t)=0$ for all $t$.
This means that $U_t$ has a multiple zero 
at $x=0$ for all $t\in\R$. Since 
$U_t$ satisfies a linear equation (\ref{eqlin}), Lemma \ref{lemzero}
implies that  $U_t\equiv0.$  This shows that  $\vp=U(\cdot,0)$ 
is a steady state.

We have thus proved that every function $\vp\in \omega(u)$ is a solution of 
the second order equation (\ref{steadyeq}), 
satisfying  $\vp(0)=\gamma$ and $\vp_x(0)=0.$ 
The uniqueness of this solution gives  $\omega(u)=\{\vp\}$, for some
$\vp$.  Condition (\ref{vpeven}) together with the description of the
solutions of (\ref{steadyeq}) given in Subsection \ref{stst} imply   
that $\vp$ is either a ground state or a constant equilibrium.
\end{proof}

For the proof of  Proposition \ref{propcontinuum}, we need 
several preliminary results. Denote
\begin{equation}
  \label{eq:5}
   J:= \{\vp(0):\vp\in\omega(u)\}.
\end{equation}
By compactness and connectedness of $\omega(u)$ in
$L^\infty_{loc}(\R)$, we have
\begin{equation}\label{Iinterval0}
J=[\gamma^-,\gamma^+] \text{  for some }\gamma^-\le \gamma^+,
\end{equation}
that is, $J$ is a
singleton or  a compact interval.
Proposition \ref{propcontinuum}
says that $\gamma^-= \gamma^+$, so this is what we want to show at the
end. First, we establish some properties of 
$\gamma^-$, $\gamma^+$.
In the formulations of the lemmas below,  
we consider the open interval
$(\gamma^-,\gamma^+)$ with the understanding that it is empty 
if $\gamma^-=\gamma^+$.   

\begin{lem}\label{lemmacase2} The following assertions hold:
 \begin{enumerate}
  \item[(i)] $f(s)>0$ for each  $s\in(\gamma^-,\gamma^+)$;
  \item[(ii)] if $\gamma^-<\gamma^+$, then 
 $\gamma^+\in\omega(u)$ and $f(\gamma^+)=0$.
 \end{enumerate}
 \end{lem}
(It is perhaps needless to say that in $\gamma^+\in\omega(u)$, $\ga^+$   
refers to the constant function taking the value $\gamma^+$.)

\begin{proof}[Proof of Lemma \ref{lemmacase2}]
We prove  (i) by contradiction. Assume there exists 
$s\in (\gamma^-,\gamma^+)$ with $f(s)\le 0$ (in particular,
$(\gamma^-,\gamma^+)\ne\emptyset$).   
 Since  $\eta(t)\to0$, the definition of $J$ implies that 
 \begin{equation}
   \label{eq:15}
   \liminf_{t\to\infty}u(\eta(t),t)=\gamma^-<s,\qquad 
\limsup_{t\to\infty}u(\eta(t),t)=\gamma^+>s.
 \end{equation}
Now, for all large $t$ the function $u(\codt,t)$ has 
a local maximum 
at $x=\eta(t)$. Therefore, 
 equation \eqref{eq1} gives
$$
   (u(\eta(t),t) )'=u_t(\eta(t),t) = u_{xx}(\eta(t),t) +
   f(u(\eta(t),t)) \le  f(u(\eta(t),t)).
$$
This and the assumption $f(s)\le 0$ imply, via an elementary comparison
argument for the equation $\dot \xi=f(\xi)$,  that if the relation
 $u(\eta(t),t) <0$ is valid for some $t$, then it remains valid for
 all larger $t$. This contradiction to \eqref{eq:15} proves statement (i).

For the proof of statement (ii), we assume that $\ga^-<\ga^+$. 
For a contradiction, we assume also that 
$\gamma^+\not\in\omega(u).$ Then, by relations \eqref{vpeven},
\eqref{monotonicity2}, and compactness of $\omega(u)$,  
there exists $\e$ with $0<\e<\gamma^+-\gamma^-$ 
such that for all $\vp\in\omega(u)$ one has $\vp(\pm 1)<\gamma^+-\e.$ 
By statement (i), $f>0$ on $(\gamma^+-\e,\gamma^+)$. 
Therefore, we can choose $s\in  (\gamma^+-\e,\gamma^+)$
such that the solution $\psi$ of (\ref{steadyeq}) with  
$\psi(0)=s$, $\psi'(0)=0$ is a nonstationary 
 periodic solution of (\ref{steadyeq}). (The existence of such $s$
follows from  Lemma \ref{MatPolLemma}, but one can also give more
direct arguments, see for example \cite[Lemma 3.2]{p-Ma:1d}).  
 Let $\rho>1$ be a period of $\psi.$ Then $-\rho$ is also a period of
 $\psi$ and we have
$$
\psi(\pm\rho) = s,\qquad \vp(\pm\rho)\le \vp(\pm1)<s \quad (\vp\in\omega(u)).
$$
Hence, there is $T_1>0$ such that
 for all $t>T_1$ one has $u(\pm\rho,t)<s$.
Therefore, by Lemma \ref{lemzero} and Remark \ref{convzero}, 
 $z_{(-\rho,\rho)}(u(\cdot,t)-\psi)$ is finite   for $t> T_1$ and
for all sufficiently large $t$
the function $u(\cdot,t)-\psi$ has only simple zeros in $(-\rho,\rho)$.  On the
other hand, the definition of $\ga^\pm$ (cp. 
\eqref{eq:5}, \eqref{Iinterval0}) yields  
$\vp\in\omega(u)$ with $\vp(0)=s=\psi(0)$. Since also $\vp'(0)=0$ (see
\eqref{vpeven}), 
$\vp-\psi$ has a multiple zero  
at $x=0.$ Applying Lemma \ref{robustnesszero} as in the proof of Lemma
\ref{nointersection}, we conclude that 
there exist sequences $t_n\to\infty,$ $x_n\to0$ such that 
$u(\cdot, t_n)-\psi$ has a multiple zero at $x=x_n,$ and we have a
contradiction. This contradiction proves that $\gamma^+\in\omega(u)$.

It remains to show that $f(\ga^+)=0$. If this is not true, then, by
statement (i),
$f(\ga^+)>0$.
Let $U$ be the entire solution of
(\ref{eq1}) with $U(\cdot,0)\equiv \ga^+$  and 
$U(\cdot,t)\in\omega(u)$ for all $t\in\R.$ 
Since $U_x$ solves a
linear equation \eqref{eqlin}, the identity $U_x(\cdot,0)\equiv 0$
implies that $U_x(\cdot,t)\equiv 0$ for all $t\in\R$. 
Thus $U=U(t)$ is a solution of 
 $U_t=f(U)$ and $U'(0)=f(\ga^+)>0$. Thus $U(t)>\ga^+$ for $t>0$, which
 contradicts the definition of $\ga^+$. The proof is now complete. 
\end{proof}

 \begin{lem}
   \label{le-consts}
If $\la\in(\ga^-,\ga^+)$, then  (the constant function) $\la$ is not 
contained in $\om(u)$.
 \end{lem}
 
 \begin{proof}
 Assuming $\la\in \om(u)$,  let $U$ be the entire solution of
(\ref{eq1}) with $U(\cdot,0)\equiv \la$  and 
$U(\cdot,t)\in\omega(u)$ for all $t\in\R.$ 
Then, as at the end of the previous proof,  
 $U=U(t)$ is a solution of $U_t=f(U),$ $U(0)=\la$. 
By Lemma \ref{lemmacase2}(i), 
the range of the function $U$ is an interval 
on  which $f>0$. We can choose $\tilde \la$ in this interval such
that the solution $\psi$ of (\ref{steadyeq}) with $\psi(0)=\tilde\la,$ 
$\psi'(0)=0$ is a nonconstant periodic solution.  This and 
symmetries of $\psi$ (cp.  (\ref{periodicorbits}))
imply that if $\rho>0$ is the minimal period of $\psi$, then 
\begin{equation*}
  \psi\lp\frac{\rho}{2}\rp=\min\psi,\qquad\psi(\rho)= 
\psi(2\rho)=\tilde \la=\max\psi.
\end{equation*}
We still have
$\tilde \la\in \om(u)$ as $\tilde \la=U\lp\tilde t\rp$ for some 
some $\tilde t$. 

Now, as $\psi-U\lp\tilde t\rp=\psi-\tilde\la$ has a multiple zero at $x=\rho$
and $\psi-U\not\equiv 0$ ($\psi$ is nonconstant), 
Lemma \ref{robustnesszero} yields sequences $t_n\to\infty,$
$x_n\to\rho$ such that 
$u(\cdot,t_n)-\psi$ has a multiple zero at $x=x_n$ for all $n.$
On the other hand, by (\ref{uxuniquezero}), 
there exists $T_2>0$ such that for all 
$t>T_2$ one has $u_x(\cdot,t)<0$ on $\lp\frac{\rho}{2},2\rho\rp.$ 
One can now obtain a contradiction by considering 
$z_{\lp\frac{\rho}{2},2\rho\rp}(u(\cdot,t)-\psi)$
and using very similar arguments as in 
 the proof of Lemma \ref{nointersection}. 
We omit the details. 
 \end{proof}

\begin{lem}\label{lemmacase22}
If  $\la\in(\gamma^-,\gamma^+)$ and the  solution $\psi$ of
\emph{(\ref{steadyeq})} with $\psi(0)=\la$, 
$\psi'(0)=0$ is periodic, then 
there is no $\vp\in\omega(u)$ such that $\vp\leq\psi$ on $\R$. 
\end{lem}
\begin{proof}
By Lemma \ref{lemmacase2}, the periodic solution 
 $\psi$ is nonconstant.
By
Lemma \ref{lemmacase2}(i)  
and (\ref{periodicorbits}), $\psi$ 
 satisfies $\max\psi=\la$ and $\min\psi<\gamma^-.$
Assume for a contradiction that there exists $\vp\in\omega(u)$ such
that $\vp\leq\psi.$ Consider the following set 
\begin{equation}
  \label{eq:8}
  K:=\{\xi\in\R: 
\text{there exists $\vp\in\omega(u)$ such that $\vp\leq \psi(\cdot-\xi)$}\}.
\end{equation}
By our assumption, $K$ contains $\xi=0$. By compactness of $\om(u)$ in
$L^\infty_{loc}(\R)$, $K$ is closed. We show
that $K$  is also open, thereby proving that actually $K=\R$. 

Fix any $\xi\in K$ and take $\vp\in \om(u)$ as in \eqref{eq:8}. 
Let $U$ be the entire solution of (\ref{eq1})  
with $U(\cdot,0)=\vp$ and $U(\cdot,t)\in\omega(u)$ 
for all $t.$ Since $ \psi(\cdot-\xi)$ is a steady state  
of (\ref{eq1}), the strong comparison argument gives 
\begin{equation}
  \label{eq:9}
  \tilde \vp:=U(\cdot,1)<\psi(\cdot-\xi).
\end{equation}
Since $\tilde \vp\in\om(u)$, relations \eqref{vpeven} and 
the periodicity of $\psi$ imply that  relation \eqref{eq:9} 
remains valid if $\xi$ is replaced by $\tilde \xi$ 
with $\tilde\xi \approx \xi$. 
This shows the openness of $K$, hence $K=\R$.

Take now $\xi$ such that   $\psi(-\xi)=\min \psi$.  For some
 $\vp\in \om(u)$ one has $\vp\le \psi(\codt-\xi)$. In
particular, 
$$\vp(0)\le \psi(-\xi)=\min \psi<\gamma^-,
$$  
which is a contradiction (cp. \eqref{eq:5}, \eqref{Iinterval0}).
This contradiction completes the proof.
\end{proof}

As a direct corollary of
Lemma \ref{lemmacase22}(i), (\ref{monotonicity2}), 
and the compactness of $\omega(u)$ in $C^1_{loc}(\R)$, we obtain the
following result:

\begin{cor}\label{corcase2}
If  $\la\in(\gamma^-,\gamma^+)$, then for each $\rho>0$ there exist
 $\kappa>0$ and $\e_1\in (0,1),$ depending on $\rho,$ such that
for any  $\vp\in\omega(u)$ with 
$\displaystyle \lb\vp(0)-\la \rb\leq\kappa$ one has
 $ \vp'<-\e_1$ 
 on $\displaystyle\lp\frac{\rho}{2},\rho\rp.$
\end{cor}

\begin{lem}\label{lemmacase23}
If $\la\in(\gamma^-,\gamma^+)$ and  the solution $\psi$ of
\emph{(\ref{steadyeq})} with $\psi(0)=\la$, 
$\psi'(0)=0$ is periodic,
 then there exist $T>0,$ $\e>0$ with the following property.
Denoting by 
$\rho>0$ is the minimal period of $\psi$, we have
\begin{equation}
  \label{eq:7}
 z_{(-\rho,\rho)}(u(\cdot,t)-\psi)\le 2
\end{equation}
whenever  $t>T$ is such that $\displaystyle
u(\pm\rho,t)\in\lp\psi(0)-\e,\psi(0)\rp$. 
\end{lem}

\begin{proof}
As in the proof of Lemma \ref{lemmacase22}, we have
 $\la=\psi(0)=\psi(\rho)=\max\psi$ and $\psi(\rho/2)=\min \psi$. 
Also recall that $u_x$ is uniformly bounded
 for $t>1.$  

With
$\kappa>0$ and $\e_1\in (0,1)$  as in Corollary \ref{corcase2}, we
define the following positive quantities:
 \begin{equation*}
  \e_2=\min\lp\frac{\kappa}{2},\frac{\rho\e_1}{8}\rp,\qquad
  \d=\frac{\e_2}{\lV u_x\rV_{L^\infty(\R\times(1,\infty))}+1}. 
 \end{equation*}
Note that $\de <\rho/8$.  In particular, $\psi(\de)<\psi(0)$.  
We will show that the conclusion of Lemma \ref{lemmacase23} is valid with
$$
\e:=\min\lp\psi(0)-\psi(\d),\e_1,\e_2\rp.
$$

First, we claim that for any 
$\vp\in \om(u)$ with $\varphi(\rho)\in [\psi(0)-\e,\psi(0)]$ one has 
$\vp(0)>\psi(0)+\e_2$. Indeed, if not, then
$ \vp(0)\le \psi(0)+\e_2<\psi(0)+\kappa$. Since also 
$\vp(\rho)\ge\psi(0)-\e>\psi(0)-\kappa$, using first
Corollary \ref{corcase2} with the mean value theorem, and then 
\eqref{monotonicity2}, we obtain  
\begin{equation*}
  \vp(\rho)\le \vp\lp\frac{\rho}2\rp-\e_1\frac{\rho}2\le
  \vp(0)-\e_1\frac{\rho}2 \le \psi(0)+\e_2-\e_1\frac{\rho}2. 
\end{equation*}
However, since $\e_2-\e_1\rho/2<-\e_2$ and $-\e_2\le- \e$, we have a
contradiction to the assumption $\varphi(\rho)\in [\psi(0)-\e,\psi(0)]$.
Thus, our claim is true.

In view of compactness of $\om(u)$ and \eqref{eq:11},
the above claim implies that 
there exists $T>1$  such that 
if $t>T$ and $u(\pm\rho,t)\in\lp\psi(0)-\e,\psi(0)\rp$, then 
$u(0,t)>\psi(0)+\e_2.$ Using our definition of $\d$ and the mean value
theorem, we infer from   
 $u(0,t)>\psi(0)+\e_2$ that $u(\cdot,t)>\psi(0)\ge\psi$ on $(-\d,\d).$ 

Next, we make $T$ larger, if necessary, so as to guarantee that if 
$t>T$ we have $u_x(\cdot,t)>0$ on $[-\rho,-\d]$ and $u_x(\cdot,t)<0$ 
on $[\d,\rho]$ (cp.  \eqref{uxuniquezero}, \eqref{maxu}).
 Thus, if $t>T$ 
and $u(\pm\rho,t)\in\lp\psi(0)-\e,\psi(0)\rp$, then for any 
$x\in [\de,{\rho}/{2}]$ we have
\begin{equation*}
  u(x,t)>u(\rho,t)>\psi(0)-\e>\psi(\de)\ge \psi(x),
\end{equation*}
where we have used the definition of $\e$ and the monotonicity of 
$\psi$ in $(0,\rho/2)$. Similarly one shows that $u(\cdot,t)>\psi$ 
on $[-{\rho}/{2},-\de]$. Combining these estimates with the previous
one, we conclude that
$u(\cdot,t)-\psi>0$ on
$[-{\rho}/{2},{\rho}/{2}]$
whenever $t>T$ and $u(\pm\rho,t)\in\lp\psi(0)-\e,\psi(0)\rp$.
Since the function  $u(\cdot,t)-\psi$ is increasing on 
$\lp-\rho,-\frac{\rho}{2}\rp$ and decreasing
on  $\lp\frac{\rho}{2},\rho\rp$, it can have at most one zero 
in each of these intervals. Thus, in total, the function
$u(\cdot,t)-\psi$ can have at most two zeros in $[-\rho,\rho]$
when $t>T$ and $u(\pm\rho,t)\in\lp\psi(0)-\e,\psi(0)\rp$), as
asserted in the lemma   (in fact,
the conclusion holds with the equality sign in \eqref{eq:7},
but this is of no significance here).   
\end{proof}

We can now complete the proof of Proposition \ref{propcontinuum}

\begin{proof}[Proof of Proposition \ref{propcontinuum}]
  We need to show that $\gamma^-<\gamma^+$. Going by contradiction, we
  assume that  $\gamma^-<\gamma^+$.
As already noted in the proof of Lemma \ref{lemmacase2}, we can then 
choose $\la \in (\gamma^-,\gamma^+)$ such that the solution 
 $\psi$ of (\ref{steadyeq}) with $\psi(0)=\la$, $\psi'(0)=0$ is
 nonconstant and periodic, and $\max \psi=\la<\ga^+$. 
Let $\rho>0$ be the minimal period of
 $\psi$ and let $\e$, $T$ be as in Lemma \ref{lemmacase23}. 
Making $\e>0$ smaller, with no effect on the conclusion of 
Lemma \ref{lemmacase23}, we may assume that 
$ \la-\e>\ga^-$.
Also, in view of \eqref{vpeven}, making $T$ larger, if necessary, 
we may assume that 
\begin{equation}
  \label{eq:12}
  |u(-\rho,t)-u(\rho,t)|<\frac{\e}{2}\quad (t>T). 
\end{equation}

We next pick 
$s\in (\gamma^-,\la-\e)$. Then there is 
$\vp\in\omega(u_0)$ with $\vp(0)=s$. 
Lemma \ref{lemmacase2} rules out the possibility that $\vp\le \psi$ in
$(-\rho,\rho)$. Therefore, using the evenness and periodicity of 
$\psi$ in conjunction with \eqref{vpeven}, \eqref{monotonicity2}, 
one shows easily that $ \psi-\vp$ has at least 4 zeros in 
$(-\rho,\rho)$.  Let now  $U$ be the  entire solution of (\ref{eq1})  
with $U(\cdot,0)=\vp$ and $U(\cdot,t)\in\omega(u)$ for all $t\in\R$.
For $t\approx 0$, we have 
$$
U(\pm \rho,t)\approx \vp(\pm\rho)<\varphi(0)=s<\psi(0)-\e=\psi(\pm\rho)-\e.
$$ 
Therefore, an application of  Lemma \ref{lemzero} shows that
arbitrarily close to 0 there is 
$t<0$ such that $z_{(-\rho,\rho)}(\psi -U(\cdot,t))\ge 4$ and all
zeros of $\psi -U(\pm \rho,t)$ in $(-\rho,\rho)$ are simple.
Replacing $\vp$ by $U(\cdot,t)$ for such $t$, 
we have thus found an element $\vp\in\omega(u_0)$ 
such that 
$$
\vp (\pm\rho)<\la-\e=\psi(0)-\e
$$ and
$ \psi-\vp$ has at least 4 simple zeros
in $(-\rho,\rho)$. 

Since $\ga^+$, $\varphi$  are elements of $\om(u_0)$, we can
approximate them arbitrarily closely in $C^1_{loc}(\R)$ by
$u(\cdot,T_1)$, $u(\cdot,T_2)$ with $T_2>T_1>T$. 
In particular, we can choose $T_2>T_1>T$ such that 
\begin{equation}
  \label{eq:13}
  u(\codt,T_1)>\psi\text{ on $[-\rho,\rho]$}
\end{equation}
and
\begin{equation}
  \label{eq:14}
  z_{(-\rho,\rho)}(\psi -u(\cdot,T_2))\ge 4,\quad  u(\pm\rho,T_2)<\psi(0)-\e.
\end{equation}
Denote
\begin{equation*}
  \tau:=\inf\{s\in (T_1,T_2]: u(\pm\rho,t)<\psi(0)-\frac{\e}2\quad
  (s\le t\le T_2)\}.
\end{equation*}
By \eqref{eq:14}, $\tau$ is a well defined element of $[T_1,T_2)$. By  
\eqref{eq:13},  $\tau>T_1$. Therefore, at least one of the values 
$u(\pm\rho,\tau)$ is equal to $\psi(0)-{\e}/2$ and the relations
$\tau>T_1\ge T$ and \eqref{eq:12} consequently give
\begin{equation*}
  u(\pm\rho,\tau) \in (\psi(0)-\e,\psi(0)).
\end{equation*}
It now follows from Lemma \ref{lemmacase23} that
\begin{equation}
  \label{eq:71}
 z_{(-\rho,\rho)}(u(\cdot,\tau)-\psi)\le 2.
\end{equation}
Since  $u(\pm\rho,t)<\psi(0)=\psi(\rho)$ on $[\tau,T_2]$
(see the the definition of $\tau$),  the monotonicity of the zero
number gives 
\begin{equation*}
   z_{(-\rho,\rho)}(u(\cdot,T_2)-\psi)\le 2,
\end{equation*}
in contradiction to \eqref{eq:14}. This contradiction shows that
$\ga^-<\ga^+$ is impossible, which completes the proof.

\end{proof}

\subsection{Case (C3): two or more limit critical points}
In this last case, we assume that 
 there exist two $C^1$ functions $\eta_1(t),$ $\eta_2(t)$ 
with $\eta_1(t)<\eta_2(t)$ and
 $\eta_i(t)\underset{t\to\infty}{\longrightarrow}\eta_i^\infty,$
$i=1,2,$ such that  and for all $t$
large enough one has
\begin{equation}\label{uxtwozeros}
 u_x(\eta_1(t),t)=u_x(\eta_2(t),t)=0.
\end{equation}
In view of Proposition \ref{loczero}, 
$\eta_1(t)<\eta_2(t)$ can be selected such that they are 
 two successive critical points  of $u(\cdot,t)$,
one of them  a local minimum point, the other one a
local maximum point. 

\begin{lem}\label{propperiodic}
Set  $\xi:=\eta_2^\infty-\eta_1^\infty$. 
If $\xi>0$, then each function  $\vp\in\omega(u)$ is  
$2\xi-$periodic.
\end{lem}
\begin{proof}
  Lemma \ref{le-conv0} implies that each function $\vp\in\omega(u)$ is
  even about each of the two distinct points $\eta_1^\infty$,
  $\eta_2^\infty$. Therefore it is also even about the points 
$2\eta_1^\infty-\eta_2^\infty$, and $2\eta_2^\infty-\eta_1^\infty$. 
Repeating such reflections arguments one obtains the $\xi$-periodicity
easily. 
\end{proof}

\begin{lem}
  \label{cor-const}
Each $\vp\in \om(u_0)$ is a constant function.
\end{lem}

\begin{proof}
Suppose first that $\eta_1^\infty=\eta_2^\infty$. Since these are the
limits of the critical points $\eta_1(t)<\eta_2(t)$, it follows that
\begin{equation*}
 \lim_{t\to\infty} u_{xx}( \eta_1^\infty,t)= \lim_{t\to\infty} u_{x}(
 \eta_1^\infty,t)=0.
\end{equation*}
Consequently, for each $\vp\in \om(u_0)$ we have $\vp_{xx}(
\eta_1^\infty)=\vp_{x}(\eta_1^\infty)=0$. Therefore, if 
$U$ is the  entire solution of (\ref{eq1})  
with $U(\cdot,0)=\vp$ and $U(\cdot,t)\in\omega(u)$ for all $t\in\R$,
we have $U_x(\eta_1^\infty,t)=U_{xx}(\eta_1^\infty,t)=0$ for all
$t\in\R$. An application of  Lemma \ref{lemzero} on a suitable
interval, one shows easily that these relations can hold only if
$U_x\equiv 0$. In particular, $\vp$ is constant.

Let now $\eta_1^\infty<\eta_2^\infty$.
  Suppose $\vp\in \om(u_0)$ and there is $\la\in \R$ with
$\vp_x(\la)\ne 0$. Then $x=\la$ is a simple zero of the function 
$V_\la(\vp)(x)=\vp(2\la-x)-\vp(x)$. Lemma \ref{propperiodic} implies 
that  $V_\la(\vp)$  is a periodic
function, hence it has infinitely many simple zeros. 
Consequently, taking $t_n\to\infty$ such that $u(\codt,t_n)\to \vp$ in
$C^1_{loc}(\R)$, we have $z(V_\la u(\cdot,t_n))\to\infty$. 
However, as noted in Subsection \ref{refl-sec},
the condition $u_0(-\infty)\ne u_0(\infty)$ implies that 
$z(V_\la u(\cdot,t))$ is finite, hence, by the monotonicity,
it is bounded from above as
$t\to\infty$.  This contradiction completes the proof. 
\end{proof}

\begin{proof}[Proof of Theorem \ref{mainthm} in the case {\rm (C2)}]
We show that $\om(u_0)=\{\vp\}$ for some constant $\vp$. 
By Lemma \ref{cor-const}---and compactness and
connectedness---$\om(u_0)$ is an interval $[a_1,b_1]$ of constants
(which we identify  with the corresponding constant functions). 
Here $a_1\le b_1$ and we want to show that $a_1=b_1$.

We go by contradiction. Assume $a_1<b_1$. Then, clearly, there are
$a<b$ such that $(a,b)\subset (a_1,b_1)$ and either  
$f \ge 0$ on $(a,b)$ or $f \le 0$ on $(a,b)$. Assume the former, the
latter is analogous. 

For large $t$, one  of points $\eta_1(t)$, $\eta_2(t)$, further
denoted by $\eta(t)$,  
is a local minimum point of $u(\cdot,t)$. From the fact that  
$\om(u_0)$ consists of constants, we infer that given any $k>0$, 
\begin{equation}
  \label{eq:10}
\sup_{x\in
  (-k,k)}|u(\eta(t),t)-u(x,t)|\underset{t\to\infty}{\longrightarrow} 0. 
\end{equation}
As $a,b\in \om(u_0)$, \eqref{eq:10} implies in particular that 
 for some sequences $t_n\to\infty$,
$t'_n\to\infty$ one has $u(\eta(t_n),t_n)\to b$,  
$u(\eta(t'_n),t'_n)\to a$. However,  
since $\eta(t)$ is a local minimum
point and  $f\ge 0$ on $(a,b)$,  equation (\ref{eq1}) gives
$$
   (u(\eta(t),t) )'=u_t(\eta(t),t) = u_{xx}(\eta(t),t) + f(u(\eta(t),t)) \ge 0,
$$
whenever $u(\eta(t),t) \in (a,b)$.
This implies that if $n$ is so large that 
$\displaystyle u(\eta(t_n),t_n)>(a+b)/2,$ then $u(\eta(t),t)>(a+b)/2$ for 
all $t>t_n$ and we have a contradiction. 

This contradiction shows that $\om(u_0)$ consists of a single constant
$\la$. The invariance of $\om(u_0)$ shows that this constant is a 
steady state of \eqref{eq1}. The proof of Theorem \ref{mainthm} is now
complete. 
\end{proof}

\bibliographystyle{plain}

\end{document}